\documentclass[11pt,oneside]{amsart}
\usepackage{geometry}                
\usepackage{graphicx}
\usepackage{amssymb}
\usepackage{epstopdf}
\usepackage{url}
\usepackage{verbatim}

\makeatletter

\renewcommand\subsubsection{\@secnumfont}{\bfseries}%
\renewcommand\subsubsection{\@startsection{subsubsection}{3}
  \z@{.5\linespacing\@plus.7\linespacing}{-.5em}%
  {\normalfont\bfseries}}
  
  \makeatother

\usepackage[shortlabels]{enumitem}

\setlist[enumerate]{topsep=0pt,partopsep=0pt}

\usepackage{amsmath,amsthm,amscd,amssymb}
\usepackage{latexsym}
\usepackage[colorlinks,citecolor=red,pagebackref,hypertexnames=false]{hyperref}
\numberwithin{equation}{section}
\theoremstyle{definition}
\newtheorem{theorem}{Theorem}[subsection]
\newtheorem{thm}[theorem]{Theorem}
\newtheorem{lemma}[theorem]{Lemma}
\newtheorem{corollary}[theorem]{Corollary}
\newtheorem{prop}[theorem]{Proposition}

\theoremstyle{definition}

\newtheorem{problem}[theorem]{Problem}
\newtheorem{example}[theorem]{Example}

\newtheorem{remark}[theorem]{Remark}

\def\RR{\mathbb{R}}
\def\CC{\mathbb{C}}

\def\QQ{\mathbb{Q}}
\def\ZZ{\mathbb{Z}}
\def\NN{\mathbb{N}}

\def\supp{\mathrm{supp}}

\newcommand{\rbr}[1]{\left( {#1} \right)}

\newcommand{\cbr}[1]{\left\{ {#1} \right\}}

\newcommand{\abs}[1]{\left| {#1} \right|}

\def\one{\mathbf{1}}

\newcommand*\wc{{}\cdot{}}

\begin{document}

\author{Thomas Cai \and Kyle Hambrook}
\title{On the Exact Fourier Dimension of Sets of Well-Approximable Matrices}
\begin{abstract}
We compute the exact Fourier dimension of the set of $\Psi$-well-approximable $m \times n$ matrices (and the set of $\Psi$-well-approximable numbers) in the homogeneous and inhomogeneous cases  
for any approximation function $\Psi$ satisfying $\sum_{q \in \ZZ^n} \Psi(q)^m < \infty$. 
\end{abstract}

\maketitle

\section{Introduction}

\subsection{Background} 

The Fourier dimension of a set $E \subseteq \RR^{d}$ is 
$$
\dim_F E = \sup\cbr{\beta \in [0,d] : \exists \mu \in \mathcal{P}(E) \text{ such that } |\widehat{\mu}(\xi)| \ll |\xi|^{-\beta/2} \quad \forall \xi \in \RR^d, \xi \neq 0}. 
$$
Here and elsewhere $\mathcal{P}(E)$ is the set of Borel probability measures on $\RR^d$ whose support is contained in $E$. 
Fourier dimension is closely related to Hausdorff dimension. 
Indeed, as a consequence of Frostman's lemma, the Hausdorff dimension of a Borel set $E \subseteq \RR^d$ is 
$$
\dim_H E = \sup\cbr{ \alpha \in [0,d] : \exists \mu \in \mathcal{P}(E) \text{ such that } \int_{\RR^d} |\widehat \mu(\xi) |^2 |\xi|^{s-d} < \infty   }. 
$$
Thus the Fourier dimension concerns the pointwise Fourier decay of measures supported on a set and the Hausdorff dimension is about the $L^2$ Fourier decay of such measures. 
More concretely, it follows that $$\dim_F E \leq \dim_H E$$ for every Borel set $E \subseteq \RR^d$.
Hausdorff and Fourier dimension are generally different. 
For $k < d$, every $k$-dimensional hyperplane in $\RR^d$ has Hausdorff dimension $k$ and Fourier dimension 0. 
The middle-$1/3$ Cantor set has Hausdorff dimension $\log 2 / \log 3$ and Fourier dimension 0. 
In fact, for every $\theta \in (0,1)$, the middle-$\theta$ Cantor set $C_{\theta}$ has $\dim_F C_{\theta} < \dim_H C_{\theta}$. 
Moreover, $0 < \dim_F C_{\theta}$ for almost every $\theta \in (0,1)$.

A set $E \subseteq \RR^d$ is called a Salem if $\dim_F E = \dim_H E$. 
Every set with Hausdorff dimension 0 is a Salem set of dimension 0. 
$\RR^d$ is a Salem set of dimension $d$, as is any subset of $\RR^d$ that contains an open ball. 
Every sphere in $\RR^d$ is a Salem set of dimension $d-1$. 
%
Salem \cite{Salem51} proved that for every $\alpha \in [0,1]$ there exists a Salem set in $\RR$ of dimension $\alpha$ using random Cantor-type sets. 
Kahane \cite{kahane-1966-brownian} (see also \cite{kahane-1966-fourier}, \cite{kahane-book}) proved that 
for every $\alpha \in [0,d]$ there exists a Salem set in $\RR^d$ of dimension $\alpha$ 
by considering images of stochastic processes.
Other random constructions of Salem sets exist (see, e.g. \cite{bluhm-1}, \cite{chen-seeger}, \cite{shmerkin-suomala}, \cite{LP2009}). 

Kaufman \cite{Kaufman} was the first to find non-random (also called explicit) examples of Salem sets with dimension different from $0, d-1,$ and $d$. 
Specifically, Kaufman showed that the set 
$$
E(\tau) = \{x \in \RR : |xq-r| < |q|^{-\tau} \text{ for infinitely many $(r,q) \in \ZZ \times \ZZ$} \}
$$
is a Salem set of dimension $2/(1+\tau)$ for all $0 < \tau < 1$. 
$E(\tau)$ is called the set of $\tau$-well-approximable numbers. 
Jarn{\'\i}k \cite{Jarnik29} and Besicovitch \cite{Bes} proved earlier that the Hausdorff dimension of $E(\tau)$ is $2/(1+ \tau)$ for $0 < \tau < 1$. 
(Note: Dirichlet's approximation theorem implies $E(\tau) = \RR$ if $\tau \geq 1$.) 

\subsection{Definition of Well-Approximable Matrices}

Our main theorem concerns the following generalization of $E(\tau)$ 
that is often studied in metric Diophantine approximation. 

Let $m,n \in \NN$. Let $Q$ be a subset of $\ZZ^n$. 
Let $\Psi : \ZZ^n \to [0,\infty)$. Let $\theta \in \RR^m$. 
Define $E(m,n,Q,\Psi,\theta)$ to be the set of all points 
$$
x = (x_{11}, \ldots, x_{1n}, \ldots, x_{m1}, \ldots, x_{mn}) \in \RR^{mn}
$$ 
such that
$$
\max_{1 \leq i \leq m} | \sum_{j=1}^{n} x_{ij} q_j - r_i - \theta_i | < \Psi(q) \text{ for infinitely many $(r,q) \in \ZZ^m \times Q$}.
$$
By identifying $x$ with the $m \times n$ matrix with $x_{ij}$ in the $i$-th row and $j$-th column, the defining inequality becomes 
$$
|xq - r - \theta| < \Psi(q) \text{ for infinitely many $(r,q) \in \ZZ^m \times Q$}.
$$
Note: In this paper $| \wc |$ denotes the max norm and $| \wc |_2$ denotes the Euclidean norm. 
The set $E(m,n,Q,\Psi,\theta)$ is called a set of $\Psi$-well-approximable matrices (or a set of $\Psi$-well-approximable linear forms).
Note $E(\tau) = E(1,1,\ZZ, |q|^{-\tau}, 0)$. 

The function $\Psi$ is called an approximation function. 
Note that redefining $\Psi$ at finitely many points does not change the set $E(m,n,Q,\Psi,\theta)$. 
Note also that for every $x \in \RR^{mn}$ and $q \in \ZZ^n$ there is an $r \in \ZZ^m$ such that $|xq - r| \leq 1/2$.
So replacing $\Psi$ by $\max\cbr{\Psi,1/2}$ does not change the set. 
Therefore, in the proofs, we can (and will) always assume $\Psi \leq 1/2$ and $\Psi(0)=1/2$. 

The set $Q$ may be called the set of denominators. 
Since $E(m,n,Q,\Psi,\theta) = E(m,n,\ZZ^n, \Psi \one_Q,\theta)$, 
it is not strictly necessary to have $Q$ explicitly in the definition, 
but it is convenient to do so. 

The cases $\theta = 0$ and $\theta \neq 0$ are called the homogeneous case and inhomogeneous case, respectively. 

\subsection{Previous Results}

Many authors have studied the Hausdorff dimension of $E(m,n,Q,\Psi,\theta)$ at various level of generality 
(see \cite{hambrook-transactions} for a short survey). 
We state only the results most relevant to our main theorem. 
%
%
%
Bovey and Dodson \cite{bovey-dodson} showed that 
$$
\dim_{H} E(m,n,\ZZ,|q|^{-\tau},0) = \min\cbr{ \dfrac{m(n-1)}{1+\tau}, mn }. 
$$
Also in the homogeneous case, Rynne \cite{RYNNE1998166} proved a completely general result. Let
$$
\eta(Q,\Psi) = \inf\{\eta \geq 0 : \sum_{\substack{q \in Q \\ q \neq 0}} |q|^m \rbr{\frac{\Psi(q)}{|q|}}^{\eta} < \infty \}. 
$$
Rynne showed that 
$$
\dim_{H} E(m,n,Q,\Psi,0) = \min\{m(n-1) + \eta(Q,\Psi),mn\}.
$$
For the inhomogeneous case, 
Theorem 1 of Beresnevich and Velani \cite{BV-slicing} implies that, if $\Psi(q) \rightarrow 0$ as $|q| \rightarrow \infty$, 
then $$\dim_{H} E(m,n,Q,\Psi,\theta) = \min\{m(n-1) + \eta(Q,\Psi),mn\}.$$

For the Fourier dimension of $E(m,n,Q,\Psi,\theta)$, the major results are the following. 
Hambrook \cite{hambrook-transactions} proved 
$$\dim_{F} E(m,n,\ZZ^n,|q|^{-\tau},\theta) \geq \min\cbr{\frac{2n}{1+\tau},mn}.$$ 
(In fact, \cite{hambrook-transactions} also proves a lower bound 
on the Fourier dimension for more general approximation functions $\Psi$ 
and sets of denominators $Q$. 
However, that result is less general and less aesthetically pleasing than the lower bound proved in the present paper.) 
Hambrook and Yu \cite{hambrook-yu} established 
the matching upper bound to obtain 
$$\dim_{F} E(m,n,\ZZ^n,|q|^{-\tau},0) = \min\cbr{\frac{2n}{1+\tau},mn}.$$  
By comparing to Bovey and Dodson's formula for the Hausdorff dimension, this proved $E(m,n,\ZZ,|q|^{-\tau},0)$ is not Salem unless $m=n=1$. 

\subsection{Main Theorem}

The following is our main result, which gives a formula for the Fourier dimension of $E(m,n,Q,\Psi, \theta)$ for all $m,n \in \NN$, all $\theta \in \RR^m$, all $Q \subseteq \ZZ^n$, 
and very general functions $\Psi$. 
It can be viewed as a Fourier dimension analog of the results of Rynne \cite{RYNNE1998166} and of Beresnevich and Velani \cite{BV-slicing} mentioned above. 
Define $\Psi_*(q)$ = $\Psi(q) |q|^{-1}$ and 
$$
s(Q,\Psi) = \inf\{ s \geq 0 : \sum_{q \in Q}  \Psi_{*}(q)^s  <  \infty \}.
$$
\begin{thm}\label{main-thm-2}
If $\sum_{q \in Q} \Psi(q)^m < \infty$, then 
$\dim_F E(m,n,Q,\Psi, \theta) = \min\{2s(Q,\Psi), mn\}$.
\end{thm}

\begin{remark}\label{divergence}
We do not have a formula for the Fourier dimension of $E(m,n,Q,\Psi, \theta)$ for all $\Psi$ when $\sum_{q \in Q} \Psi(q)^m = \infty$. 
The Fourier dimension does not necessarily equal $\min\{2s(Q,\Psi), mn\}$ in that case. 
However, we can conclude that $\dim_F E(m,n,Q,\Psi, \theta) = mn$ for certain classes of functions. 
See Section \ref{divergence-case} for details.
\end{remark}

We prove Theorem \ref{main-thm-2} by proving the following two propositions.    

\begin{prop}\label{main-upper-bound}
If $\sum_{q \in Q} \Psi(q)^m < \infty$, then 
$\dim_F E(m,n,Q,\Psi, \theta) \leq 2s(Q,\Psi)$.
\end{prop}

\begin{prop}\label{main-lower-bound}
$\dim_F E(m,n,Q,\Psi, \theta) \geq \min\cbr{2s(Q,\Psi),mn}$.
\end{prop}

Proposition \ref{main-upper-bound} is proved by generalizing the method of \cite{hambrook-yu}. 
In particular, we must prove an inhomogeneous version of the Theorem 2.3 in \cite{hambrook-yu}, 
which is the main techincal result of that paper.  
Our result is Lemma \ref{thm:linearform}; we call it the Inhomogeneos Lattice Lemma. 

Proposition \ref{main-upper-bound} is proved by modifying the argument of \cite{hambrook-transactions}. 
Significant changes are required as the hypothesis $\sum_{q \in Q} \Psi(q)^m < \infty$ 
is in a different form and more general than the hypothesis imposed on $\Psi$ and $Q$ in \cite{hambrook-transactions}. 
The key idea is the definitions of the sets $Q(M)$ and $Q'(M)$ (which are subsets of $Q$) in Subsection \ref{sec-FM-prelim}. 
The set $Q(M)$ is defined significantly differently than the set of the same name in \cite{hambrook-transactions} because 
it is adapted to the form of the hypotheses on $\Psi$ and $Q$. 
The assumption that $\sum_{q \in Q} \Psi(q)^m$ converges is critical to lower bounding the size of $Q(M)$, 
which is in turn critical to obtaining the necessary Fourier decay estimate on the measure we construct. 
The set $Q'(M)$ has no analog in \cite{hambrook-transactions}. 
It is defined by removing from $Q(M)$ the vectors $q$ that divide small non-zero vectors $\ell \in \ZZ^{mn}$. 
The idea of defining $Q'(M)$ (and the proof of lower bound on its size) is borrowed from \cite{hambrook-fraser-Rn}. 
With $Q(M)$ and $Q'(M)$ properly defined, the rest of the proof of Proposition \ref{main-lower-bound} 
is similar to the proof in \cite{hambrook-transactions}.


\section{The Divergence Case}\label{divergence-case}

When $\sum_{q \in Q} \Psi(q)^m = \infty$, the Fourier dimension of $E(m,n,Q, \Psi, \theta)$ no longer depends solely on $s(Q, \Psi)$, as the following example illustrates. 

\begin{example} Take $Q = \ZZ^n$ and consider the function 
\begin{equation*}    
\Psi(q) = 
    \begin{cases}
        1/2 & \text{if } q_1 = 2^k \text{ for some } k \in \NN \text{ and } q_2 = \cdots = q_n = 0 \\
        0 & \text{otherwise}
    \end{cases}
\end{equation*}
For every $s > 0$, $\sum_{q \in Q}  \Psi_{*}(q)^s = \sum_{k \in \NN}  (1/2)^s (2^{-ks}) < \infty$, and hence $s(Q, \Psi) = 0$. 
But $E(m,n,Q, \Psi, \theta) = \RR^{mn}$ because $\Psi(q) = 1/2$ for infinitely many $q$.
\end{example}

However, the following lemma and its corollary allows us to conclude that $\dim_F E(m,n,Q,\Psi, \theta) = mn$ under certain circumstances. 

\begin{lemma}\label{ball} Let $E \subseteq \RR^{mn}$. If $\lambda(E^c) = 0$, then $\dim_F E = mn$.
\end{lemma}

\begin{proof}
Let $B = \{x \in \RR^{mn}: |x|_2 \leq 1\}$ be the closed unit ball in $\RR^{mn}$. 
If $\lambda(E^c) = 0$ and $\lambda'$ is Lebesgue measure restricted to $E \cap B$, then $\supp(\lambda') \subseteq E$ and 
$$
\widehat{ \lambda' } (\xi) = \int e^{-2 \pi i x \cdot \xi} 1_{E \cap B}(x) dx =  \int e^{-2 \pi i x \cdot \xi} \one_{B}(x) dx = \widehat{\one_{B}}(\xi), 
$$
since the integrands agree $\lambda$-almost everywhere. But, by a standard calculation (see e.g., \cite[p.34]{mattila-book-2}), $|\widehat{\one_{B}}(\xi)| \ll |\xi|^{-(mn+1)/2}$. 
\end{proof}

For brevity, 
define 
$$
|E(m,n,Q, \Psi, \theta)| = \lambda(E(m,n,Q, \Psi, \theta) \cap [0,1]^{mn}). 
$$
Since $\lambda(E(m,n,Q, \Psi, \theta))$ is $\ZZ^{mn}$-periodic, 
$|E(m,n,Q, \Psi, \theta)| = 1$ if and only if $\lambda(E(m,n,Q, \Psi, \theta)^c) = 0$. 
Thus: 

\begin{corollary}\label{mn cor}
If $|E(m,n,Q, \Psi, \theta)| = 1$, then $\dim_F E(m,n,Q, \Psi, \theta) = mn$. 
\end{corollary}

There are several theorems in metric Diophantine approximation 
that conclude $|E(m,n,Q, \Psi, \theta)| = 1$ from the divergence of $\sum_{q \in Q} \Psi(q)^m = \infty$ (or a similar sum) and some additional hypotheses.   
We state a few of these theorems below; a good survey is \cite{BRV-2016}. 
Combining these theorems with Corollary \ref{mn cor} gives us cases where $\dim_F E(m,n,Q,\Psi, \theta) = mn$.


Khinchine's theorem \cite{Khintchine1924} says 
$$
|E(1,1,\NN,\Psi,0)|
= \begin{cases}
  0  & \text{if } \sum_{q=1}^{\infty} \Psi(q) < \infty \\
  1 & \text{if } \sum_{q=1}^{\infty} \Psi(q) = \infty \text{ and  $\Psi$ is monotone} 
\end{cases}
$$
Monotonicity is required here. Indeed, Duffin and Schaeffer \cite{Duffin1941KhintchinesPI} 
provided a counterexample where $|E(1,1,\NN,\Psi,0)| = 0$ and $\sum_{q=1}^{\infty} \Psi(q) = \infty$. 

The Duffin-Schaeffer theorem (conjectured by Duffin and Schaeffer \cite{Duffin1941KhintchinesPI} 
and proved by Koukoulopoulos and Maynard \cite{Koukoulopoulos-Maynard}) states  
$$
|E(1,1,\NN,\Psi,0)|
= \begin{cases}
  0  & \text{if } \sum_{q=1}^{\infty} q^{-1}{\phi(q)\Psi(q)} < \infty \\
  1 & \text{if } \sum_{q=1}^{\infty} q^{-1}{\phi(q)\Psi(q)} = \infty, 
\end{cases}
$$
where $\phi$ is the Euler totient function. 

If $\Psi$ is symmetric (i.e., $\Psi(q) = \psi(|q|)$ for some function $\psi: \NN \rightarrow [0,1/2]$), Allen and Ramirez \cite{allen-ramirez} proved  
$$|E(m,n,\ZZ^n,\Psi,\theta)| = \begin{cases}
  0  & \text{if }\sum_{q=1}^{\infty} q^{n-1} \Psi(q)^m < \infty \\
  1 & \text{if }\sum_{q=1}^{\infty} q^{n-1} \Psi(q)^m = \infty \text{ and  either $mn > 2$ or $\psi$ is monotone} 
\end{cases}$$
At present, this is the strongest version of the Inhomogeneous Khinchine-Groshev theorem.  
The divergence case for $n \geq 3$, the divergence case for $n=1$, $m \geq 3$, and the convergence case were proved earlier (see \cite{allen-ramirez} for details on attribution). 
Monotonicity is required when $mn=1$, as shown by the Duffin-Schaeffer counterexample mentioned above. 
It is unknown if monotonicity is necessary for $mn = 2$. 

If $\Psi(q) \rightarrow 0$ as $|q| \rightarrow \infty$, Schmidt's theorem \cite{schmidt} (see also \cite{BV-slicing}) is    
$$|E(m,n,Q,\Psi,\theta)| = \begin{cases}
  0  & \text{if } \sum_{q \in Q} \Psi(q)^m < \infty \\
  1 & \text{if } \sum_{q \in Q} \Psi(q)^m = \infty \text{ and $m+n>2$}.
\end{cases}$$
%
%
%
%
%
In the homogeneous case ($\theta = 0$), 
some stronger results are known. 
Beresnevich and Velani \cite{berenevich-velani-kg} proved 
$$
|E(m,n,Q,\Psi,0)| = 1 \quad \text{if $\sum_{q \in Q} \Psi(q)^m = \infty$ and $m>1$}. 
$$
See also Ramirez \cite{ramirez2023duffinschaeffer} for an even stronger statement.  

We are not aware of any examples where $\sum_{q \in Q} \Psi(q)^m = \infty$ and $\dim_{F} E(m,n,Q,\Psi,\theta) < mn$. 
However, it may be possible to modify the Duffin-Schaeffer counterexample to obtain such an example. Thus we propose the following problem. 

\begin{problem}
Prove or Disprove: If $\sum_{q \in Q} \Psi(q)^m = \infty$, then $\dim_{F} E(m,n,Q,\Psi,\theta) = mn$. 
\end{problem}

\section{Proof of Proposition \ref{main-upper-bound}}

\subsection{Inhomogeneous Lattice Lemma}

The following lemma is an inhomogeneous version of Theorem 2.3 in \cite{hambrook-yu}. 

Let $\theta = (\theta_1,\ldots,\theta_m) \in \RR^m$. 
For all non-zero $q=(q_1,\dots,q_n) \in \ZZ^n$ and $\delta \geq 0$, define 
	$$
	L_{\delta,q,\theta_i}=\{x \in \mathbb{R}^n: \exists r\in\mathbb{Z},|q_1x_1 + \dots + q_dx_n - r - \theta_i|\leq\delta\} 
	$$
	and
	$$
	L_{\delta,q,\theta}^m = L_{\delta,q,\theta_1} \times \cdots \times L_{\delta,q,\theta_m} \subset \RR^{mn}. 
	$$

\begin{lemma}\label{thm:linearform}
	Let $m,n \geq 1$ be integers. Let $q=(q_1,\dots,q_n) \in \ZZ^n$, $q \neq 0$.  Let $0 < \delta < 1/2$.  
	Let $\mu$ be a Borel probability measure on $[0,1]^{mn}.$ 
	Then 
	\[
	\mu(L_{\delta,q,\theta}^m) \ll \delta^m \left(1+  O(\sum_{t \in \ZZ^m, 0 < |t| \leq 2/\delta} |\hat{\mu}(t q^T )|) \right),
	\]
	and, for all $K > 0$ and $N > 0$, 
	\[
	\mu(L_{\delta,q,\theta}^m) \gg \delta^m \left(1  +  O(\sum_{t \in \ZZ^m, 0< |t| \leq K/\delta} |\hat{\mu}(t q^T )|) \right)+O(K^{-N}). 
	\]
	Moreover, the implied constant in the $O(K^{-N})$ term depends on $N$ only, and all other implied constants are absolute. 
\end{lemma}

\textbf{Note.} According to our convention for identifying $m \times n$ matrices with points in $\RR^{mn}$, 
	$$
	t q^T = 
	\begin{bmatrix}
	t_1 \\
	\vdots \\
	t_m
	\end{bmatrix}
	\begin{bmatrix}
	q_1 & \cdots & q_n \\
	\end{bmatrix}
	=
	\begin{bmatrix}
	t_1 q_1 & \cdots & t_1 q_n \\
	\vdots & & \vdots \\
	t_m q_1 & \cdots & t_m q_n
	\end{bmatrix}
	=
	(t_1 q_1, \ldots, t_1 q_n, \ldots, t_m q_1, \ldots, t_m q_n) 
	$$

\subsubsection*{Preparation for Proof of Lemma \ref{thm:linearform}} 
%
Let $q = (q_1,\ldots,q_n) \in \ZZ^n, q \neq 0$. 
Let $0 < \delta < 1/2$. 
Fix $1 \leq i \leq m$. 

%
%
For each $r \in \RR$, define 
the plane 
$$
P_{q,r, \theta_i} = \cbr{x \in \RR^n : q \cdot x - r - \theta_i = 0 },
$$
and the 
slab 
$$
P_{\delta,q,r, \theta_i} = \cbr{x \in \RR^n : |q \cdot x - r- \theta_i| \leq \delta }. 
$$
Let
$$
L_{q, \theta_i} = \bigcup_{r \in \ZZ} P_{q,r} = \cbr{x \in \RR^n : \exists r \in \ZZ : q \cdot x - r - \theta_i  = 0 }. 
$$
Note that $P_{q,r, \theta_i} = P_{0,q,r, \theta_i}$ and $L_{q, \theta_i} = L_{0,q, \theta_i}$. 
%
%

Define $\delta_{\ast} = \delta / |q|_2$.  Let $B_{\delta_{\ast}}(0)$ be the closed Euclidean metric ball with radius is $\delta_{\ast}$ centered at the origin. 
The Euclidean distance between any two planes $P_{q,r', \theta_i}$ and $P_{q,r'', \theta_i}$ is $|r' - r''| / |q|_2$. 
Therefore $L_{q, \theta_i}$ is a union of planes that are orthogonal to $q$ and spaced a distance of $1/|q|_2$ apart. 
Also,      
$P_{\delta,q,r, \theta_i}$ is the $\delta_{\ast}$-thickening of the plane $P_{q,r, \theta_i}$, i.e.,  
$$
P_{\delta,q,r, \theta_i} = P_{q,r, \theta_i} + B_{\delta_{\ast}}(0), 
$$
and $L_{\delta,q, \theta_i}$ 
is the $\delta_{\ast}$-thickening of $L_{q, \theta_i}$, i.e., 
$$
L_{\delta,q, \theta_i} = \bigcup_{r \in \ZZ} P_{\delta,q,r, \theta_i} = L_{q, \theta_i} + B_{\delta_{\ast}}(0).  
$$
Since $\delta < 1/2$, the slabs 
$P_{\delta,q,r, \theta_i}$ 
in the union 
are disjoint. 

In an abuse of notation, we use $P_{q,r, \theta_i}$ to denote 
the surface measure on the plane $P_{q,r, \theta_i}$.  
Likewise, $L_{q, \theta_i}$ denotes the surface measure on the union of planes $L_{q, \theta_i}$. The measures are related by   
$$
L_{q, \theta_i} = \sum_{r \in \ZZ} P_{q,r, \theta_i}. 
$$

Since the measure $L_{q, \theta_i}$ is (a multiple of) 
the restriction of the $(n-1)$-Hausdorff measure on $\RR^n$ to the set $L_{q, \theta_i}$, 
we have the following property: 
There are constants $a,b>0$ (independent of $q$) such that 
\begin{align}\label{LqAD}
a\epsilon^{n-1} \leq L_{q, \theta_i}(B_{\epsilon}(x)) \leq b\epsilon^{n-1} 
\end{align}
for all $0 < \epsilon < 1$ and all $x \in L_{q, \theta_i}$.

Assume (without loss of generality) that $q_1 \neq 0$. Then the measure $P_{q,r, \theta_{i}}$ is
given by
\begin{equation}\label{parameterization integral}
\int_{\RR^n}f(x) dP_{q,r,\theta_{i}} = \int_{\RR^{n-1}} f(g(x_2, \dots, x_n))\frac{1}{|q_1|}|q|_2dx_2 \dots dx_n, 
\end{equation}
where
$$
g(x_2,\ldots,x_n) = (q_1^{-1}(r + \theta_i - q_2 x_2 - \cdots - q_n x_n),x_2,\ldots,x_n)
$$ 
is the parameterization of the plane $P_{q,r, \theta_i}$.

\begin{lemma}\label{fourier Lq lemma}
For each $k \in \ZZ^n$, 
\begin{align} 
\label{fourier Lq} 
|\widehat{L_{q,\theta_i}}(k)| =
\left\{
\begin{array}{ll}
e^{-2\pi i t \theta_i}|q|_2 & \text{if  $k = tq$ for some $t \in \ZZ$}    \\
0 & \text{if  $k \notin \ZZ q$ }
\end{array}
\right.
\end{align}
\end{lemma}
\begin{proof}
We have 
$$
\widehat{L_{q, \theta_i}}{k} 
\int_{[0,1)^n} e^{-2 \pi i k \cdot x} dL_{q, \theta_i}(x) 
= \int_{[0,1)^n}  e^{-2 \pi i k \cdot x} \sum_{r \in \ZZ}  dP_{q,r, \theta_i}(x). 
$$
Note: If $r_1 = kq_1 + r_2$ for some $k \in \mathbb{Z}$, then $(x_1, \dots, x_n) \in P_{q,r_1, \theta_i}$ if and only if $(x_1 +k, \dots, x_n) \in P_{q,r_2, \theta_i}$.
%
%
Because of this periodicity property and because the distance between the planes $P_{q,r_1, \theta_i}$ and $P_{q,r_2, \theta_i}$ is $|r_1 - r_2|/|q|_2$, 
the sum over $r \in \ZZ$ can be replaced by the sum over $r=0,1,\ldots,|q_1|$. 
Thus 
$$
\widehat{L_{q, \theta_i}}{k} 
= \sum_{r = 0}^{|q_1|-1} \widehat{P_{q,r, \theta_i}}(k). 
$$
Using \eqref{parameterization integral}, we have 
\begin{align*}
&\widehat{P_{q,r, \theta_i}}(k) 
=  \int_{[0,1]^n} \exp(-2 \pi i k \cdot x) dP_{q,r, \theta_i}(x) 
\\
&=  \int_{[0,1]^{n-1}} \exp(-2 \pi i k \cdot (q_1^{-1}(r + \theta_i -x_2q_2 + \ldots -x_n q_n),x_2,\ldots,x_n))\frac{|q|}{|q_1|} dx_2 \ldots dx_n
\\
&= 
e^{-2\pi i k_1 q_1^{-1} (r+ \theta_i)}
\int_{[0,1]^{n-1}} \exp(-2 \pi i ((k_2 - k_1q_2/q_1) x_2 + \ldots (k_n - k_1q_n/q_1) x_n)) \frac{|q|}{|q_1|} dx_2 \ldots dx_n  
\\
&= e^{-2\pi i k_1 q_1^{-1} r} 
e^{-2\pi i k_1 q_1^{-1} \theta_i}
\frac{|q|}{|q_1|} 
\prod_{j=2}^{n}  \int_{[0,1]} \exp(-2\pi i (k_j - k_1 q_j/q_1) x_j) dx_j. 
\end{align*}
Therefore 
\begin{align*}
\widehat{L_{q, \theta_i}}(k) 
=  
\rbr{ \sum_{r=0}^{|q_1|-1} e^{-2\pi i k_1 q_1^{-1} r} }
e^{-2\pi i k_1 q_1^{-1} \theta_i}
\frac{|q|}{|q_1|} 
\prod_{j=2}^{n}  \int_{[0,1]} \exp(-2\pi i (k_j - k_1 q_j/q_1) x_j) dx_j. 
\end{align*}
The sum is 
\begin{align*} 
\sum_{r = 0}^{|q_1|-1} e^{-2\pi i k_1 q_1^{-1} r} =
\left\{
\begin{array}{ll}
|q_1| & \text{if  $k_1 = tq_1$ for some $t \in \ZZ$}    \\
0 & \text{if  $k_1 \notin \ZZ q_1$ }
\end{array}
\right.
\end{align*}

Suppose $k_1 = tq_1$ for some $t \in \ZZ$. Then $k_j - k_1q_j/q_1 = k_j - tq_j$ is an integer for each $j$. 
Thus the $j$-th integral equals $1$ if $k_j = tq_j$ and equals $0$ otherwise. 
So the product of integrals equals $1$ if $k_j = tq_j$ for all $j \in \{2 \dots n\}$, i.e., if $k = tq$. 
Otherwise, the product of integrals is $0$.  
\end{proof}


\subsubsection*{Proof of Lemma \ref{thm:linearform}} 
%
%
Let $q = (q_1,\ldots,q_n) \in \ZZ^n, q \neq 0$. 
Let $0 < \delta < 1/2$. 
Let $\mu$ be any Borel probability measure on $\RR^{mn}$ with support contained in $[0,1]^{mn}.$ 
For $z \in \RR^{mn}$, we write $z = (z^{(1)},\ldots,z^{(m)})$, where $z^{(i)} \in \RR^n$. 
Note that 
\begin{align}\label{indicator}
\mu(L_{\delta,q, \theta}^m) = \int_{[0,1]^{mn}} \prod_{i=1}^{m} \one_{L_{\delta,q, \theta_i}}(x^{(i)})  d\mu(x) 
\end{align}

Let $\phi$ be a non-negative Schwartz function on $\RR^n$ with $\widehat{\phi}(0)=1$. 
For $x \in \RR^n$, define 
$
\phi_{\delta_{\ast}}(x) 
= 
\phi(x/\delta_{\ast}).   
$
Note that 
$\widehat{\phi_{\delta_{\ast}}}(\xi) = \delta_{\ast}^n \widehat{\phi}(\delta_{\ast} \xi)$ 
for all $\xi \in \RR^n$. 
Note also that $\phi_{\delta_{\ast}} \ast L_{q, \theta_i}$ is a smooth $\ZZ^n$-periodic function on $\RR^n$. 

So we can achieve our desired bounds by bounding the expression in \eqref{indicator}.
To do this, we will approximate $\one_{L_{\delta,q, \theta_i}}$ 
by $\phi_{\delta_{\ast}} \ast L_{q,\theta_i}$. 
With this in mind, we first consider the integral in \eqref{indicator} with $\one_{L_{\delta,q, \theta_i}}$ replaced by $\phi_{\delta_{\ast}} \ast L_{q, \theta_i}$ for all $i \in \{1 \dots m\}$.

Let $K > 0$ be arbitrary. Recall that $\delta_{*}=\delta/|q|_2$. 
By Parseval's theorem and 
Lemma \ref{fourier Lq lemma}, 
we have 
\begin{align}
\label{parseval} 
\int_{[0,1]^{mn}} \prod_{i=1}^{m} \phi_{\delta_{\ast}} \ast L_{q,\theta_i} (x^{(i)}) d\mu(x)
&= 
\sum_{k \in \ZZ^{mn}} \overline{ \widehat{\mu}(k) } \prod_{i=1}^{m}  \widehat{L_{q,\theta_i}}(k^{(i)}) \widehat{\phi_{\delta_{\ast}}}(k^{(i)})) 
\\
\notag
&=
\delta_{\ast}^{mn} |q|_2^{m} \sum_{t \in \ZZ^m} \overline{ \widehat{\mu}( t_1q,\ldots,t_n q ) } \prod_{i=1}^{m} e^{-2 \pi i t_i \theta_i} \prod_{i=1}^{m}\widehat{\phi}(\delta_{\ast} t_i q) 
\\ 
\notag
&= \delta_{\ast}^{m(n-1)} \delta^{m} (1+ S + T), 
\end{align}
where 
\begin{align*}
S = \sum_{t \in \ZZ^m, \, 0 < |t| < K/\delta } \overline{ \widehat{\mu}( t_1q,\ldots,t_n q ) }\prod_{i=1}^{m} e^{-2 \pi i t_i \theta_i} \prod_{i=1}^{m} \widehat{\phi}(\delta_{\ast} t_i q), \\
T = \sum_{t \in \ZZ^m, \, |t| \geq K/\delta } \overline{ \widehat{\mu}( t_1q,\ldots,t_n q ) }\prod_{i=1}^{m} e^{-2 \pi i t_i \theta_i} \prod_{i=1}^{m} \widehat{\phi}(\delta_{\ast} t_i q).   
\end{align*}

Since $|\widehat{\phi}| \leq \widehat{\phi}(0)=1$, we have 
\begin{align}\label{S bound}
|S| \leq \sum_{t \in \ZZ^m, \, 0 < |t|_{\infty} < K/\delta } |\widehat{\mu}(t_1q,\ldots,t_m q) |. 
\end{align}
Since $|\widehat{\mu}| \leq \widehat{\mu}(0)=1$ 
and since $\widehat{\phi}$ is a Schwartz function, we have, for every $N > 0$, 
\begin{align}\label{T bound}
|T| 
&\ll 
\sum_{t \in \ZZ^m, \, |t|_{\infty} \geq K/\delta } \prod_{i=1}^{m} (1+ \delta_{\ast} |q|_2 |t_i| |)^{-N-m} 
\\
\notag 
&\ll 
\sum_{t \in \ZZ^m, \, |t|_{\infty} \geq K/\delta } (\delta_{\ast} |q|_2 |t|_{\infty} )^{-N-m} 
\ll
( \delta_{\ast} |q|_2 )^{-m} K^{-N}. 
\end{align}
The implied constants here depend only on $\phi$ and $N$. 


We first prove the upper bound in Theorem \ref{thm:linearform}. 
We choose $\phi$ to be a non-negative Schwartz function on $\RR^n$
such that $\widehat{\phi}(0)=1$,
$m(\phi) := \min\cbr{\phi(x) : x \in B_2(0)}$ is positive, 
and $\widehat{\phi}=0$ outside $B_K(0)$. 
For all $x \in L_{\delta,q, \theta_i} = L_{q, \theta_i} + B_{\delta_{\ast}}(0)$, there exists $z \in L_{q, \theta_i}$ with $|x-z| \leq \delta_{\ast}$, 
and hence  
$$
\phi_{\delta_{\ast}}(x-y) \geq m(\phi) \one_{B_{2\delta_{\ast}}(x)}(y) \geq m(\phi) \one_{B_{\delta_{\ast}}(z)}(y)   
$$
for all $y \in \RR^n$. 
By integrating with respect to the measure $L_{q, \theta_i}$ and then applying \eqref{LqAD}, 
it follows that 
\begin{align}\label{approx upper}
\phi_{\delta_{\ast}}  \ast L_{q,\theta_i}(x) \geq m(\phi) a \delta_{\ast}^{n-1} \one_{L_{\delta,q, \theta_i}}(x)
\end{align}
for all $x \in \RR^n$. 
Since $m(\phi) > 0$, combining this with \eqref{indicator} and \eqref{parseval} yields 
$$
\mu(L_{\delta,q, \theta}^m) \leq (m(\phi) a)^{-m} \delta^{m}(1+S+T). 
$$
Since $\widehat{\phi}=0$ outside $B_K(0)$, we have $T = 0$.  
Setting $K = 2$ and appealing to \eqref{S bound} completes the proof of the upper bound in Theorem \ref{thm:linearform}. 

To prove the lower bound in Theorem \ref{thm:linearform}, 
we put a different condition on $\phi$.
We choose $\phi$ to be a non-negative Schwartz function on $\RR^n$ 
such that $\widehat{\phi}(0)=1$ and $\phi = 0$ outside $B_1(0)$.  
For each $x$ in 
$L_{\delta,q, \theta_i}$, 
there is a $z \in L_{q, \theta_i}$ such that $|x-z| \leq \delta_{\ast}$, 
and hence 
$$
\phi_{\delta_{\ast}}(x-y) 
\leq 
|\phi|
\one_{B_{\delta_{\ast}}(x)}(y) 
\leq 
|\phi| \one_{B_{2\delta_{\ast}}(z)}(y) 
$$
for all $y \in \RR^n$.  
By integrating with respect to the measure $L_{q, \theta_i}$ and applying \eqref{LqAD}, 
$$
\phi_{\delta_{\ast}} \ast L_{q,\theta_i}(x) \leq |\phi |_{\infty} b (2\delta_{\ast})^{n-1}.  
$$ 
for all $x \in L_{\delta,q, \theta_i}$. 
On the other hand, 
if $x$ is not in $L_{\delta,q, \theta_i}$, 
we have 
$\phi_{\delta_{\ast}}(x-y) = 0$ for all $y \in L_{q, \theta_i}$, 
and so $\phi_{\delta_{\ast}} \ast L_{q,\theta_i}(x) = 0$. 
Therefore, for all $x \in \RR^n$,  
\begin{align}\label{phi lower bound 1}
\phi_{\delta_{\ast}} \ast L_{q, \theta_i}(x) \leq |\phi|_{\infty} b (2 \delta_{\ast} )^{n-1} \one_{L_{\delta,q, \theta_i}}(x). 
\end{align}
Combining this with \eqref{indicator} and \eqref{parseval} yields 
$$
\mu(L_{\delta,q, \theta}^m) \geq 2^{-m(n-1)} (|\phi | b)^{-m} \delta^{m} (1+S+T). 
$$
Since $1+S+T \geq 1-|S|-|T|$, 
applying both \eqref{S bound} and \eqref{T bound} the 
lower bound in Theorem \ref{thm:linearform}.

\subsection{Countable Stability Lemma}


The following is Lemma 2.5 of \cite{hambrook-yu}; its proof can be found there.

\begin{lemma}\label{CS}
Fourier dimension is countably stable on closed sets, i.e., for every countable collection $(A_i)$ of closed sets in $\RR^d$, 
$$
\dim_F\rbr{\bigcup_i A_i} = \sup_i \dim_F A_i. 
$$
More generally, for every countable collection $(A_i)$ of closed sets in $\RR^d$ and every set $E \subset \RR^d$, 
$$
\dim_F\rbr{E \cap \bigcup_i A_i} = \sup_i \dim_F (E \cap A_i). 
$$
\end{lemma}

We will use the following corollary. 

\begin{corollary}
\label{periodic lemma}
Let $E \subset \RR^{mn}$. 
If $E = E - k$ for all $k \in \ZZ^{mn}$, then 
$$
\dim_F (E) = \dim_F ( E \cap [0,1]^{mn} ). 
$$
\end{corollary}
\begin{proof}
For all $k \in \ZZ^{mn}$, let $A_k = [0,1]^{mn} - k$. Note $E \cap A_k = E \cap ([0,1]-k) = (E-k) \cap ([0,1]-k) = (E \cap [0,1]) - k$. Therefore $\dim_F (E \cap A_k) = \dim_F ( E \cap [0,1]^{mn})$ for all $k$.
\end{proof}

\subsection{Borel-Cantelli Argument}

To prove $\dim_F E(m,n,Q,\Psi,\theta) \leq 2s(Q,\Psi)$, we need one final lemma. 

\begin{lemma}[Borel-Cantelli Lemma - Convergence Case]
Suppose $(\Omega,\mathcal{E},P)$ is a probability space and ($E_i$) is a sequence of sets in $\mathcal{E}$ (i.e., a sequence of events). If $\sum_{i=1}^{\infty} P(E_i) < \infty$, then 
$$
P(\cbr{\omega \in \Omega : \omega \in E_i \text{ for infinitely many } i}) = 0.
$$
\end{lemma}


Corollary \ref{periodic lemma} means we only need to show 
$\dim_F(E(m,n,Q,\Psi,\theta)) \cap [0,1]^{mn}) \leq 2s(Q,\Psi)$. 
Seeking a contradiction, suppose the inequality is false. 
So there exists a $\mu \in P(E(n,d,Q,\Psi)) \cap [0,1]^{mn}$ and $s > s(Q,\Psi)$ such that 
$|\widehat \mu(\xi)| \ll |\xi|^{-s}$. 
Lemma \ref{thm:linearform} tells us, for $0 < \delta < 1/2$,
$$\mu(L_{\delta, q,\theta}) \ll \delta^m \rbr{1 + O\rbr{ \sum_{t\in \ZZ^m, 0 < |t| \leq 2/\delta}|\widehat \mu(tq^T)| }}.$$
Because of the Fourier decay assumption, the sum over $t$ is
$$\ll \sum_{t\in \ZZ^m, 0 < |t| \leq 2/\delta} |tq^T|^{-s} \leq |q|^{-s} \sum_{t\in \ZZ^m, 0 < |t| \leq 2/\delta} |t|^{-s} \ll |q|^{-s}\delta^{s-m}.$$
Therefore $\mu(L_{\delta, q,\theta}) \ll \delta^m(1 + |q|^{-s}\delta^{s-m})$. 
Setting $\delta = \Psi(q)$, we get:
$$
\sum_{q \in Q} \mu(L_{\delta, q,\theta}) 
\ll 
\sum_{q \in Q} \Psi(q)^{m} + \sum_{q \in Q} \Psi(q)^{m}|q|^{-s}\Psi(q)^{(s-m)} 
= 
\sum_{q \in Q} \Psi(q)^{m} + \sum_{q \in Q} \Psi(q)^{s}|q|^{-s} < \infty.
$$
So, because of the convergence case of the Borel-Cantelli lemma and because 
$$
E(m,n,Q,\Psi,\theta) \subseteq \{x \in \mathbb{R}^{mn} : x \in L_{\delta,q,\theta}^m \text{ for infinitely many } q \in \mathbb{Z}^n\}, 
$$
we have $\mu(E(m,n,Q,\Psi, \theta) \cap [0,1]^{mn})=0$. 
This contradicts that the support of $\mu$ is contained in $E(m,n,Q,\Psi,\theta) \cap [0,1]^{mn}$.

\section{Proof of Proposition \ref{main-lower-bound}}

\subsection{Preliminaries}\label{sec-FM-prelim}

Recall $\Psi \leq 1/2$, $\Psi_*(q)$ = $\Psi(q) |q|^{-1}$, and 
$
s(Q,\Psi) = \inf\{ s \geq 0 : \sum_{q \in Q}  \Psi_{*}^s(q)  <  \infty \}. 
$ 
It suffices to show that $\dim_F E(m,n,Q,\Psi, \theta) \geq 2s$ for all $s < s(Q,\Psi)$. 
To do this, we fix $s < s(Q,\Psi)$ arbitrarily and construct a probability measure 
$\mu$ such that $\supp(\mu) \subseteq E(m,n,Q,\Psi,\theta)$ and 
$|\widehat{\mu}(\xi)| = o(|\xi|^{-s})$ as $\xi \to \infty$. 

For each $M > 0$, define 
$$
Q(M) =  \left\{q \in Q : M/2 < \Psi^{-s}_*(q) \leq M \right\}.  
$$

Define 
$$
\mathcal{M} = \cbr{  2^k :  k \in \NN, \sum_{q \in Q(2^k)} \Psi_{*}(q)^s \geq k^{-(n+1)}  }. 
$$

\begin{lemma}\label{Q-M-lemma} 
$\mathcal{M}$ is unbounded and 
for every $M \in \mathcal{M}$  
$$
|Q(M)| \geq \dfrac{M}{2 \log_2^{n+1}(M)}.  
$$
\end{lemma}
\begin{proof}
Since $s < s(Q,\Psi)$, we have $\sum_{q\in Q} \Psi^s_*(q) = \infty$. 
Splitting the sum dyadically gives  $\sum_{k=0}^{\infty}  \sum_{q \in Q(2^k)} \Psi_{*}(q)^s = \infty$. 
Then $\sum_{q \in Q(2^k)} \Psi_{*}(q)^s \geq k^{-(n+1)}$ for infinitely many integers $k \geq 0$.  
Now assume $M \in \mathcal{M}$. Then $M = 2^{k}$ for some $k \in \NN$, $k^{n+1} = \log_2^{n+1}(M)$, and 
$$
|Q(M)| \max_{q \in Q(M)} \Psi_{*}(q)^s \geq \sum_{q \in Q(M)} \Psi_{*}(q)^s \geq \dfrac{1}{\log_2^{n+1}(M)}. 
$$
By definition, $2/M > \Psi_{*}(q)^s$ for each $q \in Q(M)$. 
\end{proof}

For each $\ell = (\ell_{11},\ldots,\ell_{1n},\ldots,\ell_{m1},\ldots,\ell_{mn}) \in \ZZ^{mn}$, 
define $\ell_j = (\ell_{1j},\ldots,\ell_{mj})$ for all $1 \leq j \leq n$,
and define 
\begin{align*}
D(\ell)  
&= \cbr{q \in \ZZ^n : \text{There exists $k \in \ZZ^m$ such that } \ell_{ij} = q_j k_i \text{ for every  $i$,$j$} } 
\\ 
&= \cbr{q \in \ZZ^n : \exists k \in \ZZ^m, \ell = k q^T}.  
\end{align*}

If $mn=1$, $D(\ell)$ is the set of divisors 
of the integer $\ell$. 
Thus, for general $m,n$, we view $D(\ell)$ as the set of divisors of the element $\ell \in \ZZ^{mn}$. 
We will establish a divisor bound for $D(\ell)$.

For $s,t > 0$, define 
\begin{align*}
w_s(t) = \exp\rbr{  \frac{s \log t}{\log \log t} }. 
\end{align*}
For every positive integer $\ell$, define $\tau(\ell)$ to be the number of positive integers that divide $\ell$. 
The function $\tau$ obeys the following classic divisor bound, 
which is due to Wigert. 
For the proof, we refer the reader to 
Wigert \cite{Wigert} or Hardy and Wright \cite[Theorem 317]{HW}. 

\begin{lemma}\label{integer divisor bound}
For every integer $\ell \geq 3$ and every $\zeta > \log 2$, there exists $C_\zeta \in \RR$ such that 
$$
\tau(\ell) 
\leq C_\zeta w_{\zeta}(\ell). 
$$
\end{lemma}

\begin{lemma}\label{divisor bound}
For every $\ell \in \ZZ^{mn}$, $|\ell| \geq 3$ and every $\zeta > \log 2$, there exists $C_\zeta \in \RR$ such that
$$
|D(\ell)| \leq C_\zeta w_{\zeta}(|\ell|)
$$
\end{lemma}
\begin{proof}
Fix $\ell \in \ZZ^{mn}$. Choose $i_0 \in \cbr{1 \dots m}$ and 
$j_0 \in \cbr{1 \dots n}$ such that 
$|\ell| = |\ell_{j_0}| = |\ell_{i_0, j_0}|$.
We claim that the map $(q_1,\ldots,q_n) \mapsto q_{j_0}$ is an injection  
from $D(\ell)$ 
to the set of integers that divide $|\ell| = |\ell_{i_0 j_0}|$. 
To see this, note that if $q \in D(\ell)$, then there is $k_{i_0} \in \ZZ$ such that 
$\ell_{i_0,j} = q_j k_{i_0}$ for all $j$. In particular, 
$\ell_{i_0,j_0} = q_{j_0} k_{i_0} \neq 0$. 
So we can solve the first equation for $k_{i_0}$ and then solve the second equation for $q_j$ to get $q_j = q_{j_0} \ell_{i_0,j} / \ell_{i_0,j_0}$. 
The desired result now follows from Lemma \ref{integer divisor bound}. 
\end{proof}

We define a new set $Q'(M)$ by removing from $Q(M)$ those $q$ which divide small non-zero 
elements $\ell \in \ZZ^{mn}$.  
This is needed for Lemma \ref{FM-Fourier-Bounds}(c) below. 
We also show that this requires removing only a small number of elements, 
which is important for Lemma \ref{FM-Fourier-Bounds}(d) below. 

For each $M \geq 2$, define 
$$
L(M) = \cbr{\ell \in \ZZ^{mn} : 0 < |\ell| \leq \rbr{\dfrac{M}{2 \log_2^{n+1}(M)}}^{1/2mn} },  
$$
$$
Q'(M) = Q(M) \setminus \bigcup_{\ell \in L(M)} D(\ell). 
$$

\begin{lemma}\label{Q-M-prime-lemma}
There is a number $M_0$ such that for all $M \in \mathcal{M}$ with $M \geq M_0$, 
$$|Q'(M)| \geq \dfrac{M}{4 \log_2^{n+1}(M)}.$$ 
\end{lemma}
\begin{proof}
Let $M \geq 2$. 
Note $|L(M)| \leq 2^{mn} \rbr{M/(2\log_2^{n+1}(M))}^{1/2}$. 
By Lemma \ref{divisor bound}, 
for every $\ell \in L(M)$ and every $\zeta > \log 2$ we have 
$$
|D(\ell)| \leq C_{\zeta} w_{\zeta}(|\ell|) \leq C_{\zeta} w_{\zeta}\rbr{ \rbr{\dfrac{M}{2 \log_2^{n+1}(M)}}^{1/2mn} }. 
$$

Thus, for every $\ell \in L(M)$, we have $|D(\ell)| = \rbr{\dfrac{M}{2 \log_2^{n+1}(M)}}^{o(1)}$. 
Then 
$$
\abs{ \bigcup_{\ell \in L(M)} D(\ell) } = 2^{mn} \rbr{\dfrac{M}{2 \log_2^{n+1}(M)}}^{1/2 + o(1)}
$$
Therefore, by Lemma \ref{Q-M-lemma}, if $M \in \mathcal{M}$, then 
$$
|Q'(M)| \geq \dfrac{M}{2 \log_2^{n+1}(M)} -  2^{mn} \rbr{\dfrac{M}{2 \log_2^{n+1}(M)}}^{1/2 + o(1)}
$$
\end{proof}

\subsection{Periodic Bump Function}\label{sec-bump}




Fix a positive integer $K$ such that $K > mn + s$. 
Fix an arbitrary function $\phi: \RR^m \rightarrow \RR$ such that 
$\phi \geq 0$, 
$\text{supp}(\phi) \subseteq (-1,1)^m$, 
$\int_{\RR^m} \phi(x) dx = 1$, and 
\begin{align}\label{phi-decay}
|\widehat{\phi}(\xi)| \ll (1+|\xi|)^{-K} \quad \text{ for all $\xi \in \RR^m$.} 
\end{align}

For every $\epsilon > 0$ and $x \in \RR^m$, define 
\begin{align*}
\phi^{\epsilon}(x) &= \epsilon^{-m}\phi(\epsilon^{-1}x), \\
\Phi^{\epsilon}(x) &= \sum_{r \in \ZZ^m} \phi^{\epsilon}(x-r). 
\end{align*}
Note that $\Phi^{\epsilon}$ 
is periodic for the lattice $\ZZ^m$. 
Since $\text{supp}(\phi) \subseteq [-1,1]^m$, for each fixed $x \in \RR^m$, 
the sum defining $\Phi^{\epsilon}(x)$ has at most finitely many non-zero terms.  

Recall $\theta \in \RR^m$. 
For $\epsilon > 0$, $q \in \ZZ^n$, and $x=(x_{11},\ldots,x_{1n},\ldots,x_{m1},\ldots,x_{mn}) \in \RR^{mn}$, define $xq$ by identifying $x$ with the $m \times n$ matrix whose $ij$-entry is $x_{ij}$, and define 
\begin{align*}
\Phi^{\epsilon}_{q,\theta}(x) = \Phi^{\epsilon}(xq - \theta)
\end{align*}
Note that $\Phi^{\epsilon}_{q,\theta}$ 
is periodic for the lattice $\ZZ^{mn}$.

Recall the definition of $D(\ell)$ from Section \ref{sec-FM-prelim}. 

\begin{lemma}\label{Phi-fourier}
Let $\epsilon > 0$. 
Let $\ell = (\ell_{11},\ldots,\ell_{1n},\ldots,\ell_{m1},\ldots,\ell_{mn}) \in \ZZ^{mn}$. 
Let $q \in \ZZ^n$. 
For every $j \in \cbr{1,\ldots,n}$ for which $q_{j} \neq 0$, 
\begin{align*}
\widehat{\Phi^{\epsilon}_{q,\theta}}(\ell) 
=
\left\{
\begin{array}{ll}
e^{-2 \pi i q_{j}^{-1} \ell_{j} \cdot \theta} \widehat{\phi}( \epsilon q_{j}^{-1} \ell_{j} ) & \text{if } q \in D(\ell),\\
0 & \text{otherwise, }
\end{array} \right.
\end{align*}
where $\ell_j = (\ell_{1j},\ldots,\ell_{mj})$. 
\end{lemma}

\begin{proof}
By direct computation, 
\begin{align*}
\widehat{\Phi^{\epsilon}}(k) = \widehat{\phi^{\epsilon}}(k) = \widehat{\phi}(\epsilon k) \quad \forall k \in \ZZ^m. 
\end{align*}
It follows that $\int_{[0,1]^{m}} |\Phi^{\epsilon}(x)| dx = \int_{\RR^m} \phi(x) dx  < \infty$ and $\sum_{k \in \ZZ^m} |\widehat{\Phi^{\epsilon}}(x)| = \sum_{k \in \ZZ^m} |\widehat{\phi}(\epsilon k)| < \infty$. 
Thus, Fourier inversion (see, e.g., \cite[p.169]{grafakos}) gives 
\begin{align*}
\Phi^{\epsilon}(x) = \sum_{k \in \ZZ^m} \widehat{\phi}(\epsilon k) e^{ 2 \pi i k x } \quad \text{for a.e. } x \in \RR^{m} 
\end{align*}
and the series on the right  converges uniformly to a continuous function. 
Therefore, for each $q \in \ZZ^n$, 
\begin{align*}
\Phi^{\epsilon}_{q,\theta}(x) = \Phi^{\epsilon}(qx-\theta) = \sum_{k \in \ZZ^m} \widehat{\phi}(\epsilon k) e^{ 2 \pi i k \cdot (xq - \theta) } \quad \text{for a.e. } x \in \RR^{mn} 
\end{align*}
and the series on the right converges uniformly to a continuous function. 
By uniform convergence, 
\begin{align*}
\widehat{\Phi^{\epsilon}_{q,\theta}}(\ell) 
&=
\sum_{k \in \ZZ^m}  \int_{[0,1]^{mn}} \widehat{\phi}(\epsilon k) e^{2 \pi i k \cdot (xq - \theta)} e^{- 2 \pi i \ell \cdot x} dx \\
\notag
&=
\sum_{k \in \ZZ^m}  e^{- 2 \pi i k \cdot \theta } \widehat{\phi}(\epsilon k) \int_{[0,1]^{mn}} e^{2 \pi i (k \cdot (xq - \theta) - \ell \cdot x)} dx \\
\notag
&=
\sum_{k \in \ZZ^m}  e^{- 2 \pi i k \cdot \theta } \widehat{\phi}(\epsilon k) \prod_{i=1}^{m} \prod_{j=1}^{n} \int_{[0,1]} e^{2 \pi i x_{ij}(k_i q_j - \ell_{ij})} dx_{ij}. \\
\notag
\end{align*}
If $k_i q_j = \ell_{ij}$ for every $i,j$, the product of the integrals is $1$. 
Otherwise, the product is $0$. Whenever $q_j$ is non-zero, $k_i = q_j^{-1} \ell_{ij}$ for every $i$, i.e., $k = q_j^{-1} \ell_j$. 
\end{proof}

\subsection{Single-Scale Function}\label{sec-FM}

For $M > 0$ and $x \in \RR^{mn}$, define 
\begin{align*}
F_M(x) = \frac{1}{|Q'(M)|} \sum_{q \in Q'(M)} \Phi^{\Psi(q)}_{q,\theta}(x). 
\end{align*}
Note $F_{M}$ 
is periodic for the lattice $\ZZ^{mn}$.

\begin{lemma}\label{FMhat lemma}
For every $M > 0$ and $\ell \in Z^{mn}$, $\ell \neq 0$, 
\begin{align*}
\widehat{F_{M}}(\ell) = \frac{1}{|Q'(M)|} \sum_{q \in Q'(M) \cap D(\ell)} e^{-2 \pi i q_{1}^{-1} \ell_{1} \cdot \theta} \widehat{\phi}( \Psi(q) q_{j_0}^{-1} \ell_{j_0} )
\quad \forall \ell \in \ZZ^{mn}.
\end{align*}
where $j_0$ is chosen so that $|\ell| = |\ell_{j_0}|$.
\end{lemma}
\begin{proof}
Immediate from Lemma \ref{Phi-fourier}. 
\end{proof}

For part (d) of the next lemma, recall the definition of $M_0$ from Lemma \ref{Q-M-prime-lemma}.

\begin{lemma}\label{FM-Fourier-Bounds} \hspace{1pt}
\hspace{1pt}
\begin{enumerate}[(a)]
\item $\widehat{F_M}(0) = 1$ for every $M > 0$ 

\item $|\widehat{F_M}(\ell)| \leq 1$ for every $M > 0$, $\ell \in \ZZ^{mn}$ 

\item $\widehat{F_M}(\ell) = 0$ for every $M \geq 2$, $\ell \in \ZZ^{mn}$, $0 < |\ell| \leq \rbr{\dfrac{M}{2 \log_2^{n+1}(M)}}^{1/2mn}$ 

\item $|\widehat{F_{M}}(\ell)| \leq C_\zeta |\ell|^{-s} w_{\zeta}(|\ell|) \log^{n+1}(M)$ for every $M \in \mathcal{M}$, $M \geq M_0$, $\ell \in \ZZ^{mn}$, $|\ell| \geq 3$, $\zeta > \log 2$. 
\end{enumerate}
\end{lemma}

\begin{proof}\hspace{1pt}
\begin{enumerate}[(a)]
\item Since $\widehat{\phi}(0) = \int_{\RR^m} \phi(x) dx = 1$ and $D(0) = \ZZ^n$, Lemma \ref{Phi-fourier} implies (a). 
\item Since $\phi \geq 0$, we have $F_{M} \geq 0$, which gives $|\widehat{F_M}(\ell)| \leq \widehat{F_M}(0)$. Then (a) implies (b). 
\item Because of the definition of $Q'(M)$, 
if $0 < |\ell| \leq \rbr{{M}/{(2 \log_2^{n+1}(M))}}^{1/2mn}$, 
then the sum over $q$ in Lemma \ref{FMhat lemma} 
is empty; hence, $\widehat{F_M}(\ell)=0$. 

\item  Let  $\ell \in \ZZ^{mn}$, $|\ell| \geq 3$. Let $M \in \mathcal{M}$, $M \geq M_0$. By Lemma \ref{FMhat lemma} and \eqref{phi-decay}, 
\begin{align*}
\notag |\widehat{F_{M}}(\ell)| 
&\leq \frac{1}{|Q'(M)|} \sum_{q \in Q'(M) \cap D(\ell)} |\widehat{\phi}(\Psi(q) q_{j_0}^{-1} \ell_{j_0})| \\
\notag 
&\leq  \frac{C_\zeta}{|Q'(M)|}  \sum_{q \in Q'(M) \cap D(\ell)}  (1 + \Psi(q) |q|^{-1} |\ell| )^{-K} \\
\notag 
&\leq \frac{C_\zeta}{|Q'(M)|}  \sum_{q \in Q'(M) \cap D(\ell)}  (1 + \Psi(q) |q|^{-1} |\ell| )^{-s} \\
\notag 
&=  \frac{C_\zeta}{|Q'(M)|}  \sum_{q \in Q'(M) \cap D(\ell)} (1 + \Psi_{*}(q) |\ell| )^{-s} \\
\notag
& \leq |\ell|^{-s} \frac{C_\zeta}{|Q'(M)|} \sum_{q \in Q'(M) \cap D(\ell)} \frac{1}{\Psi_{*}(q)^s} \\ 
\end{align*}
For all $q \in Q(M)$, we have $1/{\Psi^s_*(q)} \leq M$, so 

$$
|\widehat{F_{M}}(\ell)| \leq |\ell|^{-s} \frac{C_\zeta}{|Q'(M)|} |D(\ell)| M
$$
The desired result now follows from Lemma \ref{divisor bound} and Lemma \ref{Q-M-prime-lemma}. 
\end{enumerate}
\end{proof}

\begin{lemma}\label{support lemma} \hspace{1pt}
\begin{enumerate}[(a)]
\item For every $M > 0$, 
\begin{align*}
\supp(F_M) \subseteq 
\bigcup_{q \in Q'(M)} \bigcup_{r \in \ZZ^m} 
\{ x \in \RR^{mn} : |xq-r-\theta|  \leq \Psi(q)   \}. 
\end{align*}
\item For every sequence $(M_k)$ that satisfies $0 < 2M_k \leq M_{k+1}$ for all $k \in \NN$, 
\begin{align*}
\bigcap_{k=1}^{\infty} \supp(F_{M_k}) \subseteq E(m,n,Q',\Psi,\theta). 
\end{align*}
\end{enumerate}
\end{lemma}
\begin{proof}
Observe that, for all $x \in \RR^{mn}$, 
\begin{align*}
F_{M}(x) = \frac{1}{|Q'(M)|} \sum_{q \in Q'(M)} \sum_{r \in \ZZ^m} \Psi(q)^{-m}\phi(\Psi(q)^{-1}(xq - r - \theta)). 
\end{align*}
Since $\supp(\phi) \subseteq (-1,1)^m$, there is some constant $0 < c < 1$ such that $\supp(\phi) \subseteq [-c,c]^m$. 
If $F_M(x)>0$, then there is some $q \in Q'(M)$ and some $r \in \ZZ^m$ such that $\phi(\Psi(q)^{-1}(xq - r - \theta)) > 0$. 
Hence $|xq-r-\theta| \leq c\Psi(q)$ 
and   
$
|r| 
\leq c\Psi(q) + |\theta| + |x| |q|_1.
$ 
Now suppose $x' \in \supp(F_M)$. 
Choose a sequence $(x_k)$ that converges to $x'$ and satisfies $F_M(x_k) > 0$ for all $k$. 
Since $(x_k)$ converges, there is some constant $C>0$ such that $|x_k| \leq C$ for all $k$. Thus $(x_k)$ is contained in the set 
$$
\bigcup_{q \in Q'(M)} \bigcup_{\substack{ r \in \ZZ^m \\ |r| \leq c\Psi(q) + |\theta| + C|q|_1 } } \{ x \in \RR^{mn} : |xq-r-\theta|  \leq c\Psi(q)   \}
$$
Since this set is closed, it contains $x'$. 
This proves (a). 
If $(M_k)$ satisfies $2M_k \leq M_{k+1}$ for all $k$, 
then the sets $Q'(M_k)$ are pairwise disjoint, 
so (b) follows from (a). 
\end{proof}

\subsection{Stability Lemma}\label{main lemma}

The next lemma will 
let us pass the Fourier decay of the function $F_M$ to the measure $\mu$.

\begin{lemma}\label{main-lemma}
Define
$$
g(\xi) = \begin{cases}
1 & \text{if } |\xi| \leq 3, \\
|\xi|^{-s} w_1(|\xi|)\log^{n+1}(|\xi|)  & \text{if } |\xi| > 3.
\end{cases}
$$
%
For every $\delta > 0$, $M_0 > 0$, and $\chi \in C^{K}_{c}(\RR^{mn})$, there is an 
$M_{\ast}(\delta, M_0, \chi) \in \mathcal{M}$ 
and $C_{\chi} >0$ such that $M_{\ast} \geq M_0$ and 
\begin{align*}
|\widehat{\chi F_{M_{\ast}}}(\xi) - \widehat{\chi}(\xi)| \leq \delta g(\xi)  \quad \forall \xi \in \RR^{mn}. 
\end{align*} 
\end{lemma}
Note: The proof will show $M_{\ast}$ can be taken to be any sufficiently large element of $\mathcal{M}$.

\begin{proof}
We start with two facts. 
First, since $\chi \in C^{K}_{c}(\RR^{mn})$, there is a $C_{\chi} > 0$ such that
\begin{align}\label{108}
|\widehat{\chi}(\xi)| \leq C_{\chi}(1+|\xi|)^{-K} \quad \forall \xi \in \RR^{mn}.
\end{align}
Second, for every $p > mn$, we have
\begin{align}\label{108-2}
\sup_{\xi \in \RR^{mn}} \sum_{\ell \in \ZZ^{mn}} (1+|\xi - \ell|)^{-p} < \infty.
\end{align}

Now we expand the expression $\widehat{\chi F_{M_{\ast}}}(\xi) - \widehat{\chi}(\xi)$.  
Since $F_M$ is $C^{K}$ and $\ZZ^{mn}$-periodic, we have
$$
F_M(x) = \sum_{\ell \in \ZZ^{mn}} \widehat{F_M}(\ell) e^{2 \pi i \ell \cdot x} \quad \forall x \in \RR^{mn}
$$
with uniform convergence. 
Since $\chi \in L^1(\RR^{mn})$, multiplying by $\chi$ and taking the Fourier transform yields 
\begin{align*}
\widehat{\chi F_M}(\xi) 
= \sum_{\ell \in \ZZ^{mn}} \widehat{F_M}(\ell) \int_{\RR^{mn}} \chi(x) e^{2 \pi i (\ell - \xi) \cdot x} dx
= \sum_{\ell \in \ZZ^{mn}} \widehat{F_M}(\ell) \widehat{\chi}(\xi-\ell)
\end{align*}
for all $\xi \in \RR^{mn}$.
Then by Lemma \ref{FM-Fourier-Bounds} (a) and (c) we have 
\begin{align}\label{110-2}
\widehat{\chi F_{M}}(\xi) - \widehat{\chi}(\xi)
=
\sum_{\ell \in \ZZ^{mn}} \widehat{\chi}(\xi-\ell) \widehat{F_M}(\ell) - \widehat{\chi}(\xi) 
=
\sum_{|\ell| > h} \widehat{\chi}(\xi-\ell) \widehat{F_M}(\ell). 
\end{align}%

Define $h = \rbr{{M}/{(2 \log_2^{n+1}(M))}}^{1/2mn}$. Note $h \rightarrow \infty$ as $M \rightarrow \infty$.
Define $\eta = (K-mn-s)/2$. Recall that $K$ was chosen to be greater than $mn+s$, so $\eta$ is positive.
To estimate $\widehat{\chi F_{M}}(\xi) - \widehat{\chi}(\xi)$, 
we 
consider two cases.

\textbf{Case 1:} Assume $|\xi| < h/2$. If $|\ell| > h$, then $|\xi - \ell| > h/2 > |\xi|$. Hence by Lemma \ref{FM-Fourier-Bounds}(b), \eqref{108},\eqref{108-2}, and \eqref{110-2} we have
\begin{align*}
| \widehat{\chi F_{M}}(\xi) - \widehat{\chi}(\xi) |
&\leq
C_{\chi} \sum_{|\ell| > h} (1+|\xi - \ell|)^{-K}
= 
C_{\chi} \sum_{|\ell| > h} (1+|\xi - \ell|)^{-s-\eta-(mn+\eta)} \\
&\leq
C_{\chi} (1 + |\xi|)^{-s} (1 +h/2)^{-\eta} \sum_{|\ell| > h} (1+|\xi - \ell|)^{-(mn+\eta)} 
\leq
\delta g(\xi)
\end{align*}
for all sufficiently large M. 

\textbf{Case 2:} Assume $|\xi| \geq h/2$. 
Thus $(2|\xi|)^{2mn} \geq M$. 
Using \eqref{110-2}, write 
$$
\widehat{\chi F_{M}}(\xi) - \widehat{\chi}(\xi) 
= \sum_{\substack{|\ell| > h \\ |\ell| \leq |\xi|/2}} \widehat{\chi}(\xi-\ell)  \widehat{F_M}(\ell) 
+ \sum_{\substack{|\ell| > h \\ |\ell| > |\xi|/2}} \widehat{\chi}(\xi-\ell)  \widehat{F_M}(\ell) 
= S_1 + S_2.
$$
If $|\ell| \leq |\xi|/2$, then $|\xi - \ell| \geq |\xi|/2 \geq h/4$. 
Hence by Lemma \ref{FM-Fourier-Bounds}(b), \eqref{108}, and \eqref{108-2} we have
\begin{align*}
|S_1|
&\leq
C_{\chi} \sum_{\substack{|\ell| > h \\ |\ell| \leq |\xi|/2}} (1+|\xi - \ell|)^{-K} 
= C_{\chi} \sum_{\substack{|\ell| > h \\ |\ell| \leq |\xi|/2}} (1+|\xi - \ell|)^{-s-\eta-(mn+\eta)} \\
&\leq
C_{\chi} (1 + |\xi|/2)^{-s} (1 + h/4)^{-\eta} \sum_{\substack{|\ell| > h \\ |\ell| \leq |\xi|/2}} (1+|\xi - \ell|)^{-(mn+\eta)} 
\leq 
\frac{\delta}{2} g(\xi)
\end{align*}
for all sufficiently large $M$. 

%
%
Fix $\log 2 < \zeta < 1.$ Note that $|\ell|^{-s} w_\zeta(|\ell|)$ is an eventually decreasing function. 
By \ref{FM-Fourier-Bounds}(d), \eqref{108}, and \eqref{108-2},  
we have
\begin{align*}
|S_2|
&\leq
C_{\zeta} C_{\chi} \sum_{\substack{|\ell| > h \\ |\ell| > |\xi|/2}} |\ell|^{-s} w_{\zeta}(|\ell|) \log^{n+1}(M)  (1+|\xi - \ell|)^{-K} \\
&\leq
C_{\zeta} C_{\chi} (|\xi|/2)^{-s} w_{\zeta}(|\xi|/2) \log^{n+1}((2 |\xi|)^{2mn}) \sum_{\substack{|\ell| > h \\ |\ell| > |\xi|/2}}  (1+|\xi - \ell|)^{-K} \\
&\leq \frac{\delta}{2}(|\xi|)^{-s} w_{1}(|\xi|) \log^{n+1}(|\xi|) = \frac{\delta}{2} g(\xi)
\end{align*}
for all sufficiently large $M \in \mathcal{M}$.

\end{proof}

\subsection{The Measure $\mu$}\label{completing the proof}

To complete the proof of Proposition \ref{main-lower-bound}, the measure $\mu$ with the desired Fourier decay and support properties (see Subsection \ref{sec-FM-prelim}) 
is constructed the same way as in \cite{hambrook-transactions}.
Arbitrarily fix $f_0 \in C_{c}^{K}(\RR^{mn})$ with $\int_{\RR^{mn}} f_0(x)dx = 1$, $\text{supp}(f_0) = [-1,1]^{mn}$, and $f_0(x) > 0$ for all $|x|<1$. 
Use Lemma \ref{main-lemma} to define numbers 
$$
M_1 = M_{\ast}(2^{-1},1,f_0), \quad M_{k} = M_{\ast}(2^{-k-1},2M_{k-1},f_0F_{M_1}\cdots F_{M_{k-1}}) \quad \forall k \in \NN, k \neq 1.
$$
and measures 
$$
d\mu_0 = f_0 dx, \quad d\mu_k = f_0  F_{M_1} \cdots F_{M_{k}} dx \quad \forall k \in \NN
$$
such that 
\begin{align*}
|\widehat{\mu_k}(\xi) - \widehat{\mu_{k-1}}(\xi)| \leq 2^{-k-1} g(\xi) \quad \forall \xi \in \RR^{mn}, k \in \NN. 
\end{align*}
L\'{e}vy's continuity theorem implies  
$(\mu_k)_{k=0}^{\infty}$ converges weakly 
(i.e., in distribution) to a finite non-zero Borel measure $\mu$ satisfying 
$|\widehat{\mu}(\xi)| = o(|\xi|^{-s})$ as $\xi \to \infty$. 
Multiplying by an appropriate constant, makes $\mu$ a probability measure. 
Lemma \ref{support lemma} implies $\supp(\mu) \subseteq E(m,n,Q,\Psi,\theta)$. 
Since there is nothing new about the construction, we omit further the details.

\bibliographystyle{myplain}
\bibliography{Master}
\end{document}